\documentclass[12pt]{amsart}
\usepackage{amsfonts,amsmath,latexsym,a4wide,amssymb,amscd,amsthm,amsxtra,mathrsfs,amsrefs,graphicx,color}

\usepackage[bookmarksnumbered, colorlinks, plainpages]{hyperref}

\hypersetup{colorlinks=true,linkcolor=red, anchorcolor=green,
citecolor=cyan, urlcolor=red, filecolor=magenta, pdftoolbar=true}

\usepackage{upref}
\usepackage{txfonts}
\usepackage{graphicx}

\usepackage[width=2cm, height=2cm, left=2cm, right=2cm, top=2cm, bottom=1cm]{geometry}

\allowdisplaybreaks

\newtheorem{theorem}{Theorem}[section]

\newtheorem{corollary}[theorem]{Corollary}
\theoremstyle{definition}
\newtheorem{definition}[theorem]{Definition}

\newtheorem{convention}[theorem]{Convention}

\theoremstyle{remark}
\newtheorem{remark}[theorem]{Remark}
\numberwithin{equation}{section}

\def\loc{\operatorname{loc}}

\def\esup{\operatornamewithlimits{ess\,sup}}

\def\R{\mathbb R}
\def\Z{\mathbb Z}
\def\ap{\approx}

\def\qq{\qquad}
\def\rn{\R^n}

\def\b{q}

\def\up{\upsilon}
\def\la{\lambda}

\def\vp{\varphi}
\def\s{\sigma}

\def\i{\infty}

\def\I{(0,\i)}

\def\rw{\rightarrow}
\def\up{\uparrow}
\def\dn{\downarrow}

\def\ls{\lesssim}

\def\R{\mathbb R}

\def\M{\mathfrak M}

\def\mp{{\mathfrak M}}
\def\W{{\mathcal W}}

\begin{document}

\setcounter{page}{1}

\title[Iterated Hardy-type inequalities involving suprema]{Iterated Hardy-type inequalities involving suprema}

\author[A. Gogatishvili]{Amiran Gogatishvili}
\address{Institute of Mathematics \\
Academy of Sciences of the Czech Republic \\
\v Zitn\'a~25 \\
115~67 Praha~1, Czech Republic} \email{gogatish@math.cas.cz}

\author[R.Ch.Mustafayev]{Rza Mustafayev}
\address{Department of Mathematics \\ Faculty of Science and Arts \\ Kirikkale
University \\ 71450 Yahsihan, Kirikkale, Turkey}
\email{rzamustafayev@gmail.com}

\thanks{The research of A. Gogatishvili was partly supported by the grants P201-13-14743S
of the Grant Agency of the Czech Republic and RVO: 67985840, by
Shota Rustaveli National Science Foundation grants no. 31/48
(Operators in some function spaces and their applications in Fourier
Analysis) and no. DI/9/5-100/13 (Function spaces, weighted
inequalities for integral operators and problems of summability of
Fourier series). The research of both authors was partly supported
by the joint project between  Academy of Sciences of Czech Republic
and The Scientific and Technological Research Council of Turkey}

\subjclass[2010]{Primary 26D10; Secondary 26D15.}

\keywords{quasilinear operators, iterated Hardy inequalities,
weights.}

\begin{abstract}
In this paper the complete solution of the restricted inequalities
for supremal operators are given. The boundedness of the composition
of supremal operators with the Hardy and Copson operators in
weighted Lebesgue spaces are characterized.
\end{abstract}

\maketitle


\section{Introduction}\label{in}

Throughout the paper we assume that $I : = (a,b)\subseteq (0,\i)$.
By $\mp (I)$ we denote the set of all measurable functions on $I$.
The symbol $\mp^+ (I)$ stands for the collection of all $f\in\mp
(I)$ which are non-negative on $I$, while $\mp^+ (I;\dn)$ and $\mp^+
(I;\up)$ are used to denote the subset of those functions which are
non-increasing and non-decreasing on $I$, respectively. When $I = \I$, we write simply $\mp^{\dn}$ and $\mp^{\up}$ instead of $\mp^+ (I;\dn)$ and $\mp^+
(I;\up)$, accordingly. The family of all weight functions (also called just weights) on $I$, that is,
locally integrable non-negative functions on $\I$, is given by $\W(I)$.

For $p\in (0,\i]$ and $w\in \mp^+(I)$, we define the functional
$\|\cdot\|_{p,w,I}$ on $\mp (I)$ by
\begin{equation*}
\|f\|_{p,w,I} : = \left\{\begin{array}{cl}
\left(\int_I |f(x)|^p w(x)\,dx \right)^{1/p} & \qq\mbox{if}\qq p<\i \\
\esup_{I} |f(x)|w(x) & \qq\mbox{if}\qq p=\i.
\end{array}
\right.
\end{equation*}

If, in addition, $w\in \W(I)$, then the weighted Lebesgue space
$L^p(w,I)$ is given by
\begin{equation*}
L^p(w,I) = \{f\in \mp (I):\,\, \|f\|_{p,w,I} < \i\}
\end{equation*}
and it is equipped with the quasi-norm $\|\cdot\|_{p,w,I}$.

When $w\equiv 1$ on $I$, we write simply $L^p(I)$ and
$\|\cdot\|_{p,I}$ instead of $L^p(w,I)$ and $\|\cdot\|_{p,w,I}$,
respectively.

Given a operator $T:\mp^+ \rw \mp^+$, for $0 < p <\i$ and
$u\in\mp^+$, denote by
$$
T_u (g) : = T (g u), \qq g \in \mp^+.
$$

Suppose $f$ be a measurable a.e. finite function on ${\mathbb R}^n$.
Then its non-increasing rearrangement $f^*$ is given by
$$
f^* (t) = \inf \{\lambda > 0: |\{x \in {\mathbb R}^n:\, |f(x)| >
\lambda \}| \le t\}, \quad t \in (0,\infty),
$$
and let $f^{**}$ denotes the Hardy-Littlewood maximal function of
$f$, i.e.
$$
f^{**}(t) : = \frac{1}{t} \int_0^t f^* (\tau)\,d\tau, \quad t > 0.
$$
Quite many familiar function spaces can be defined using the
non-increasing rearrangement of a function. One of the most
important classes of such spaces are the so-called classical Lorentz
spaces.

Let $p \in (0,\infty)$ and $w \in {\mathcal W}$. Then the classical
Lorentz spaces $\Lambda^p (w)$ and $\Gamma^p (w)$ consist of all
functions $f \in {\mathfrak M}$ for which  $\|f\|_{\Lambda^p(w)} <
\infty$ and $\|f\|_{\Gamma^p(w)} < \infty$, respectively. Here it is
$$
\|f\|_{\Lambda^p(w)} : = \|f^*\|_{p,w,(0,\infty)} \qquad \mbox{and}
\qquad \|f\|_{\Gamma^p(w)} : = \|f^{**}\|_{p,w,(0,\infty)}.
$$
For more information about the Lorentz $\Lambda$ and $\Gamma$ see
e.g. \cite{cpss} and the references therein.

The Hardy and Copson operators are defined by
$$
H g (t) : = \int_0^t g(s)\,ds, \qq g \in \mp^+,
$$
and
$$
H^* g (t) : = \int_t^{\i} g(s)\,ds, \qq g \in \mp^+,
$$
respectively. The operators  $H$ and $H^*$ play a prominent role. There are other operators that are also of interest. For example, certain
specific problems such as the description of the behaviour of the
fractional maximal operator on classical Lorentz spaces \cite{ckop}
or the optimal pairing problem for Sobolev imbeddings  \cite{kerp}
or various questions arising in the interpolation theory can be
handles in an elegant way with the help of the supremal operators

$$
S g (t) : = \esup_{0 < \tau \le t} g(\tau), \qq g \in \mp^+,
$$
and
$$
S^* g (t) : = \esup_{t \le \tau < \infty} g(\tau), \qq g \in \mp^+.
$$

In this paper, we give complete characterization of restricted inequalities
\begin{equation}\label{supr.ineq.1}
\| S_u (f)\|_{q,w,\I} \le c\|f\|_{p,v,\I},~ f \in \M^{\dn},
\end{equation}
\begin{equation}\label{supr.ineq.3}
\| S_u (f) \|_{q,w,(0,\i)} \le c \| f \|_{p,v,(0.\i)}, \, f \in \M^{\up}
\end{equation}
and
\begin{equation}\label{supr.ineq.2}
\| S_u^* (f)\|_{q,w,\I} \le c\|f\|_{p,v,\I},~ f \in \M^{\up},
\end{equation}
\begin{equation}\label{supr.ineq.4}
\| S_u^* (f)\|_{q,w,\I} \le c\|f\|_{p,v,\I},~ f \in \M^{\dn}.
\end{equation}
Note that inequality \eqref{supr.ineq.1} was characterized in \cite{gop}. It should be mentioned here that it was done under some additional condition on weight function $u$, when $q < p$ (cf. \cite[Theorem 3.4]{gop}).

In particular, we characterize the validity of the iterated
Hardy-type inequalities involving suprema
\begin{equation}\label{ISI.1}
\left\|S_{u} \bigg( \int_0^x h\bigg)\right\|_{q,w,\I} \leq c
\,\|h\|_{p,v,\I},
\end{equation}
and
\begin{equation}\label{ISI.2}
\left\|S_{u} \bigg( \int_x^{\i} h\bigg)\right\|_{q,w,\I} \leq c
\,\|h\|_{p,v,\I},
\end{equation}
where $0 < q < \infty$, $1 \le p < \i$, $u$, $w$ and $v$ are
weight functions on  $(0,\infty)$.

It is worth to mentoin that the characterizations of "dual"
inequalities
\begin{equation}\label{ISI.3}
\left\| S_{u}^* \bigg( \int_x^{\i} h\bigg)\right\|_{q,w,\I} \leq c
\,\|h\|_{p,v,\I},
\end{equation}
and
\begin{equation}\label{ISI.4}
\left\|S_{u}^* \bigg( \int_0^x h\bigg)\right\|_{q,w,\I} \leq c
\,\|h\|_{p,v,\I},
\end{equation}
can be easily obtained  from the solutions of inequalities
\eqref{ISI.1} - \eqref{ISI.2}, respectively, by change of variables. Note that inequality
\eqref{ISI.4} has been characterized in \cite{gop} in the case $0 < q < \infty$, $1 \le p < \infty$.

We pronounce that the characterizations of inequalities
\eqref{ISI.1} - \eqref{ISI.4} are important because many
inequalities for classical operators  can be reduced to them (for
illustrations of this important fact, see, for instance,
\cite{GogMusPers2}). These inequalities play an important role in
the theory of Morrey spaces and other topics (see \cite{BGGM1},
\cite{BGGM2} and \cite{BO}).

The fractional maximal operator, $M_{\gamma}$, $\gamma \in (0,n)$, is defined at $f \in L_{\loc}^1(\rn)$ by
$$
(M_{\gamma} f) (x) = \sup_{Q \ni x} |Q|^{ \gamma / n - 1} \int_{Q} |f(y)|\,dy,\quad x \in \rn,
$$
where the supremum is extended over all cubes $Q \subset \rn$ with sides parallel to the coordinate axes. It was shown in \cite[Theorem 1.1]{ckop} that
\begin{equation}\label{frac.max op.eq.1.}
(M_{\gamma}f)^* (t) \ls \sup_{t \le \tau < \infty} \tau^{\gamma / n - 1} \int_0^{\tau} f^*(y)\,dy \ls  (M_{\gamma} \tilde{f})^* (t),
\end{equation}
for every $f \in L_{\loc}^1(\rn)$ and $t \in \I$, where $\tilde{f} (x) = f^* (\omega_n |x|^n)$ and $\omega_n$ is the volume of $S^{n-1}$.
Thus, in order to characterize boundedness of the fractional maximal operator $M_{\gamma}$ between classical Lorentz spaces it is necessary and sufficient to characterize the validity of the weighted inequality
$$
\bigg( \int_0^{\infty} \bigg[\sup_{t \le \tau < \infty} \tau^{\gamma / n - 1} \int_0^{\tau} \vp(y)\,dy\bigg]^q w(t)\,dt\bigg)^{1 / q} \ls \bigg( \int_0^{\infty} [\vp(t)]^p v(t)\,dt\bigg)^{1 / p}
$$
for all $\vp \in \M^{\dn}$. This last estimate can be interpreted as a restricted weighted inequality for the operator $T_{\gamma}$, defined by
\begin{equation}\label{frac.max op.eq.3.}
(T_{\gamma}g) (t) = \sup_{t \le \tau < \infty} \tau^{\gamma / n - 1} \int_0^{\tau} g(y)\,dy, \quad g \in \M^+ \I, \quad t \in \I.
\end{equation}
Such a characterization was obtained in \cite{ckop} for the particular case when $1 < p \le q <\infty$ and in \cite[Theorem 2.10]{o} in the case of more general operators and for extended range of $p$ and $q$. Full proofs and some further extensions and applications can be found in \cite{edop}, \cite{edop2008}.

The operator $T_{\gamma}$ is a typical example of what is called a Hardy-operator involving suprema
$$
(T_u g)(t) : = \sup_{t \le s < \infty} \frac{u(s)}{s} \int_0^s g(y)\,dy,
$$
which combines both the operations (integration and taking the supremum).

In the above-mentioned applications, it is required to characterize a restricted weighted inequality for $T_u$. This amounts to finding a necessary and sufficient condition on a triple of weights $(u,\,v,\,w)$ such that the inequality
\begin{equation}\label{ineq.for.Tu}
\bigg(\int_0^{\infty} \bigg(\sup_{t \le s < \infty} u(s) f^{**}(s)\bigg)^q w(t)\,dt\bigg)^{1 / q} \ls \bigg(\int_0^{\infty} f^*(t)^p v(t)\,dt \bigg)^{1/p}
\end{equation}
holds. Particular examples of such inequalities were studied in \cite{ckop} and, in a more systematic way, in \cite{gop}. Inequality \eqref{ineq.for.Tu} was investigated in \cite{gogpick2007} in the case when $0< p \le 1$. The approach used in this paper was based on a new type reduction theorem which showed connection between three types of restricted weighted inequalities.

Rather interestingly, such operators have been recently encountered in various research projects. They have been found indispensable in the search for optimal pairs of rearrangement-invariant norms for wich a Sobolev-type inequality holds (cf. \cite{kerp}). They constitute a very useful tool for characterization of the assocaiate norm of an operator-induced norm, which naturally appears as an optimal domain norm in a Sobolev embedding (cf. \cite{pick2000}, \cite{pick2002}). Supremum operators are also very useful in limitimg interpolation theory as can be seen from their appearance for example in \cite{evop}, \cite{dok}, \cite{cwikpys}, \cite{pys}.

\begin{definition}\label{Tub.defi.1}
    Let $u \in \W\I \cap C\I$, $b \in \W\I$ and $B(t) : = \int_0^t b(s)\,ds$. Assume that $b$
    be such that $0 < B(t) < \infty$ for every $t \in \I$. The operator $T_{u,b}$ is defined at $g \in \M^+ \I$ by
    $$
    (T_{u,b} g)(t) : = \sup_{t \le \tau < \infty} \frac{u(\tau)}{B(\tau)} \int_0^{\tau} g(s)b(s)\,ds,\qquad t \in \I.
    $$
\end{definition}
The operator $T_{\gamma}$, defined in \eqref{frac.max op.eq.3.}, is a particular example of operators $T_{u,b}$. These operators are investigated in \cite{gop} and \cite{gogpick2007}.

In this paper we give complete characterization for the inequality
\begin{equation}\label{Tub.thm.1.eq.1}
\|T_{u,b}f \|_{q,w,\I} \le c \| f \|_{p,v,\I}, \qq f \in \M^{\dn}(0,\i)
\end{equation}
for $0 < q \le \infty$, $0 < p < \infty$ (see Theorems \ref{Tub.thm.1} and \ref{RT.SO.thm.3}).

Inequality \eqref{Tub.thm.1.eq.1} was characterized in \cite[Theorem 3.5]{gop} under additional condition
$$
\sup_{0 < t < \infty} \frac{u(t)}{B(t)} \int_0^t \frac{b(\tau)}{u(\tau)}\,d\tau < \infty.
$$
Note that the case when $0 < p \le 1 < q < \infty$ was not considered in \cite{gop}. It is also worse to mention that in the case when $1 < p < \infty$, $0 < q < p < \infty$, $q \neq 1$ \cite[Theorem 3.5]{gop} contains only discrete condition. In \cite{gogpick2007} the new reduction theorem was obtained when $0 < p \le 1$, and this technique allowed to characterize inequality \eqref{Tub.thm.1.eq.1} when $b \equiv 1$, and in the case when $0 < q< p \le 1$ this paper contains only discrete condition.

The paper is organized as follows. Section \ref{pre} contains some
preliminaries along with the standard ingredients used in the
proofs. Full characterization of inequalities  \eqref{supr.ineq.1} - \eqref{supr.ineq.4} and \eqref{ISI.1} - \eqref{ISI.3} are given in Sections \ref{s.3} and \ref{s.4}. Finally, solution of inequality \eqref{Tub.thm.1.eq.1} are obtained in Section \ref{s.5}.

\section{Notations and Preliminaries}\label{pre}

Throughout the paper, we always denote by  $c$ or $C$ a positive
constant, which is independent of main parameters but it may vary
from line to line. However a constant with subscript such as $c_1$
does not change in different occurrences. By $a\lesssim b$,
($b\gtrsim a$) we mean that $a\leq \la b$, where $\la >0$ depends on
inessential parameters. If $a\lesssim b$ and $b\lesssim a$, we write
$a\approx b$ and say that $a$ and $b$ are  equivalent.  We will denote by $\bf 1$ the
function ${\bf 1}(x) = 1$, $x \in \I$. Unless a
special remark is made, the differential element $dx$ is omitted
when the integrals under consideration are the Lebesgue integrals.
Everywhere in the paper, $u$, $v$ and $w$ are weights.

We need the following notations:
$$
\begin{array}{ll}
V(t) : = \int_0^t v, & V_*(t) : = \int_t^{\i} v,\\ [10pt] W(t) : =
\int_0^t w, & W_*(t) :  = \int_t^{\i} w.
\end{array}
$$

\begin{convention}\label{Notat.and.prelim.conv.1.1}
    We adopt the following conventions:

    {\rm (i)} Throughout the paper we put $0 \cdot \i = 0$, $\i / \i =
    0$ and $0/0 = 0$.

    {\rm (ii)} If $p\in [1,+\i]$, we define $p'$ by $1/p + 1/p' = 1$.

    {\rm (iii)} If $0 < q < p < \infty$, we define $r$ by $1 / r  = 1 / q - 1 / p$.

    {\rm (iv)} If $I = (a,b) \subseteq \R$ and $g$ is monotone
    function on $I$, then by $g(a)$ and $g(b)$ we mean the limits
    $\lim_{x\rw a+}g(x)$ and $\lim_{x\rw b-}g(x)$, respectively.
\end{convention}

We recall some reduction theorems for positive monotone operators from \cite{GogStep} and \cite{GogMusIHI}. The following conditions will be used below:

{\rm (i)} $T(\la f) = \la Tf$ for all $\la \ge 0$ and $f \in \M^+$;

{\rm (ii)} $Tf(x) \le c Tg(x)$ for almost all $x \in \R_+$ if $f(x)
\le g(x)$ for almost all $x \in \R_+$, with constant $c > 0$
independent of $f$ and $g$;

{\rm (iii)} $T(f+ \la {\bf 1}) \le c (T f + \la T {\bf 1})$ for all
$f \in \M^+$ and $\la \ge 0$, with a constant $c > 0$ independent of
$f$ and $\la$.

\begin{theorem}[\cite{GogStep}, Theorem 3.1]\label{Reduction.Theorem.thm.3.1}
    Let $0 < q \le \infty$ and $1 \le p < \infty$, and let $T: {\mathfrak M}^+ \rightarrow {\mathfrak M}^+$
    be an operator. Then the inequality
    \begin{equation}\label{Reduction.Theorem.eq.1.1}
    \|Tf \|_{q,w,(0,\infty)} \le c \| f \|_{p,v,(0,\infty)}, \qquad f \in {\mathfrak M}^{\downarrow}(0,\infty)
    \end{equation}
    implies the inequality
    \begin{equation}\label{Reduction.Theorem.eq.1.2}
    \bigg\| T\bigg( \int_x^{\infty} h \bigg)\bigg\|_{q,w,(0,\infty)} \le c \| h
    \|_{p, V^{p} v^{1-p},(0,\infty)}, \qquad h \in {\mathfrak M}^+(0,\infty).
    \end{equation}
    If $V(\infty) = \infty$ and if $T$ is an operator satisfying
    conditions {\rm (i)-(ii)}, then the condition
    \eqref{Reduction.Theorem.eq.1.2} is sufficient for inequality
    \eqref{Reduction.Theorem.eq.1.1} to hold on the cone ${\mathfrak M}^{\downarrow}$.
    Further, if $0 < V(\infty) < \infty$, then a sufficient condition for
    \eqref{Reduction.Theorem.eq.1.1} to hold on ${\mathfrak M}^{\downarrow}$ is that both
    \eqref{Reduction.Theorem.eq.1.2} and
    \begin{equation}\label{Reduction.Theorem.eq.1.3}
    \|T {\bf 1}\|_{q,w,(0,\infty)} \le c \|{\bf 1}\|_{p,v,(0,\infty)}
    \end{equation}
    hold in the case when $T$ satisfies the conditions {\rm (i)-(iii)}.
\end{theorem}

\begin{theorem}[\cite{GogStep}, Theorem 3.2]\label{Reduction.Theorem.thm.3.2}
    Let $0 < q \le \infty$ and $1 \le p < \infty$, and let $T: {\mathfrak M}^+ \rightarrow {\mathfrak M}^+$
    satisfies conditions {\rm (i)} and {\rm (ii)}. Then a sufficient
    condition for inequality \eqref{Reduction.Theorem.eq.1.1} to hold is
    that
    \begin{equation}\label{Reduction.Theorem.eq.1.222}
    \bigg\| T\bigg( \frac{1}{V^2(x)} \int_0^x hV \bigg)\bigg\|_{q,w,(0,\infty)}
    \le c \| h \|_{p, v^{1-p},(0,\infty)}, \qquad h \in {\mathfrak M}^+(0,\infty).
    \end{equation}
    Moreover, \eqref{Reduction.Theorem.eq.1.1} is necessary for
    \eqref{Reduction.Theorem.eq.1.222} to hold if conditions {\rm
        (i)-(iii)} are all satisfied.
\end{theorem}

\begin{theorem}[\cite{GogStep}, Theorem 3.3]\label{Reduction.Theorem.thm.3.3}
    Let $0 < q \le \infty$ and $1 \le p < \infty$, and let $T: {\mathfrak M}^+ \rightarrow {\mathfrak M}^+$
    be an operator. Then the inequality
    \begin{equation}\label{Reduction.Theorem.eq.1.1.00}
    \|Tf \|_{q,w,(0,\infty)} \le c \| f \|_{p,v,(0,\infty)}, \qquad f \in {\mathfrak M}^{\uparrow}(0,\infty)
    \end{equation}
    implies the inequality
    \begin{equation}\label{Reduction.Theorem.eq.1.2.00}
    \bigg\| T\bigg( \int_0^x h \bigg)\bigg\|_{q,w,(0,\infty)} \le c \| h \|_{p,
        V_*^{p} v^{1-p},(0,\infty)}, \qquad h \in {\mathfrak M}^+(0,\infty).
    \end{equation}
    If $V_*(0) = \infty$ and if $T$ is an operator satisfying the
    conditions {\rm (i)-(ii)}, then the condition
    \eqref{Reduction.Theorem.eq.1.2.00} is sufficient for inequality
    \eqref{Reduction.Theorem.eq.1.1.00} to hold. If $0 < V_*(0) < \infty$
    and $T$ is an operator satisfying the conditions {\rm
        (i)-(iii)}, then \eqref{Reduction.Theorem.eq.1.1.00} follows from
    \eqref{Reduction.Theorem.eq.1.2.00} and
    \eqref{Reduction.Theorem.eq.1.3}.
\end{theorem}

\begin{theorem}[\cite{GogStep}, Theorem 3.4]\label{Reduction.Theorem.thm.3.4}
    Let $0 < q \le \infty$ and $1 \le p < \infty$, and let $T: {\mathfrak M}^+ \rightarrow {\mathfrak M}^+$
    satisfies conditions {\rm (i)} and {\rm (ii)}. Then a sufficient
    condition for inequality \eqref{Reduction.Theorem.eq.1.1.00} to hold
    is that
    \begin{equation}\label{Reduction.Theorem.eq.1.2222}
    \bigg\| T\bigg( \frac{1}{V_*^2(x)} \int_x^{\infty} hV_*
    \bigg)\bigg\|_{q,w,(0,\infty)} \le c \| h \|_{p, v^{1-p},(0,\infty)}, \qquad h \in
    {\mathfrak M}^+(0,\infty).
    \end{equation}
    Moreover, \eqref{Reduction.Theorem.eq.1.1.00} is necessary for
    \eqref{Reduction.Theorem.eq.1.2222} to hold if conditions {\rm
        (i)-(iii)} are all satisfied.
\end{theorem}

\begin{theorem}\cite[Theorem 3.1]{GogMusIHI}\label{RT.thm.main.3}
    Let $0 < \b \le \i$, $1 < p < \i$, and let $T: \M^+ \rw \M^+$
    satisfies conditions {\rm (i)-(iii)}.  Assume that $u,\,w \in \W\I$
    and $v \in \W\I$ be such that
    \begin{equation}\label{RT.thm.main.3.eq.0}
    \int_0^x v^{1-p^{\prime}}(t)\,dt < \i \qq \mbox{for all} \qq x > 0.
    \end{equation}
    Then inequality
    \begin{equation}
    \label{IHI.H.1} \left\|T \bigg(\int_0^x
    h\bigg)\right\|_{q,w,(0,\infty)} \leq c
    \,\|h\|_{p,v,(0,\infty)}, \, h \in {\mathfrak M}^+,
    \end{equation}
    holds iff
    \begin{equation}\label{RT.thm.main.3.eq.2}
    \| T_{\Phi^2} f \|_{\b, w, \I} \le c \| f \|_{p, \phi,\I}, \, f \in
    \M^{\dn},
    \end{equation}
    holds, where
    $$
    \phi (x) \equiv \phi\big[v;p\big](x) : = \bigg( \int_0^x
    v^{1-{p}^{\prime}}(t)\,dt\bigg)^{- {p^{\prime}} / {(p^{\prime} + 1)}}
    v^{1-{p}^{\prime}}(x)
    $$
    and
    $$
    \Phi(x) \equiv \Phi \big[v;p\big](x) = \int_0^x \phi(t)\,dt = \bigg( \int_0^x
    v^{1-{p}^{\prime}}(t)\,dt\bigg)^{{1} / {(p^{\prime} + 1)}}.
    $$
\end{theorem}

\begin{theorem}\cite[Theorem 3.11]{GogMusIHI}\label{RT.thm.main.8.0}
    Let $0 < \b \le \i$, and let $T: \M^+ \rw \M^+$ satisfies conditions {\rm (i)-(iii)}.  Assume that $u,\,w \in \W\I$
    and $v \in \W\I$ be such that $V(x) < \infty$ for all $x > 0$.
    Then inequality
    \begin{equation}    \label{RT.thm.main.8.0.eq.1}
    \left\|T \bigg(\int_0^x h\bigg)\right\|_{\b,w,\I} \leq c
    \,\|h\|_{1,V^{-1},\I}, \, h \in \M^+,
    \end{equation}
    holds iff
    \begin{equation}\label{RT.thm.main.8.0.eq.2}
    \| T_{V^2} f \|_{\b, w, \I} \le c \| f  \|_{1, v,\I}, \, f \in
    \M^{\dn}.
    \end{equation}
\end{theorem}



\section{Supremal operators on the cone of monotone functions}\label{s.3}

In this section, we give complete characterization of  inequalities \eqref{supr.ineq.1} - \eqref{supr.ineq.4}.

To state the next statements we need the following notations:
$$
\overline{u}(t) : = \sup_{0 < \tau \le t} u(\tau), \qquad  \underline{u}(t) : = \sup_{t \le \tau < \infty} u(\tau), \qquad (t>0).
$$
For a given weight $v$, $0 \le a < b \le \infty$ and $1 \le p < \infty$, we denote
$$
\s_p (a,b) = \begin{cases}
\bigg ( \int\limits_a^b [v(t)]^{1-p'}dt\bigg)^{1 / p'} & \qquad \mbox{when} ~ 1 < p < \infty \\
\esup\limits_{a < t < b} \, [v(t)]^{-1} & \qquad \mbox{when} ~ p = 1.
\end{cases}
$$

Recall the following theorem.
\begin{theorem}\cite[Theorems 4.1 and 4.4]{gop}\label{supr.thm.101}
    Let $1 \le p < \infty$, $0 < q < \infty$ and let $u \in \W\I \cap C\I$. Assume that $v,\,w \in \W\I$
    be such that
    $$
    0 < V(x) < \i \qquad \mbox{and} \qquad  0 < W(x) < \infty \qq \mbox{for all} \qq x > 0.
    $$
    Then inequality \eqref{ISI.4}
    is satisfied with the best constant $c$ if and only if:

    {\rm (i)} $p \le q$, and in this case $c \ap A_1$, where
    $$
    A_1: = \sup_{x > 0}\bigg( [\underline{u}]^q(x)  W(x) + \int_x^{\infty}  [\underline{u}]^q(t) w(t)\,dt\bigg)^{1 / q}\s_p(0,x);
    $$

    {\rm (ii)} $q < p$,  and in this case $c \ap B_1 + B_2$, where
    \begin{align*}
    B_1: & = \bigg(\int_0^{\infty} \bigg(\int_x^{\infty}  [\underline{u}]^q(t)  w(t)\,dt\bigg)^{r / p} [\underline{u}]^q(x) \bigg[\s_p(0,x)\bigg]^r w(x)\,dx \bigg)^{1 / r}, \\
    B_2: & = \bigg(\int_0^{\infty} W^{r / p}(x) \bigg[\sup_{x \le \tau < \infty} \underline{u}(\tau) \s_p (0,\tau) \bigg]^r  w(x)\,dx \bigg)^{1 / r}.
    \end{align*}
\end{theorem}

Using change of variables $x = 1/t$, we can easily obtain the following statement.
\begin{theorem}\label{supr.thm.111}
    Let $1 \le p < \infty$, $0 < q < \infty$ and let $u \in \W\I \cap C\I$. Assume that $v,\,w \in \W\I$
    be such that
    $$
    0 < V_*(x) < \i \qquad \mbox{and} \qquad  0 < W_*(x) < \infty \qq \mbox{for all} \qq x > 0.
    $$
    Then inequality \eqref{ISI.2}
    is satisfied with the best constant $c$ if and only if:

    {\rm (i)} $p \le q$, and in this case $c \ap A_1$, where
    $$
    A_1: = \sup_{x > 0}\bigg( [\overline{u}]^q(x) W_*(x) + \int_0^x  [\overline{u}]^q(t)  w(t)\,dt\bigg)^{1 / q}\s_p(x,\infty);
    $$

    {\rm (ii)} $q < p$,  and in this case $c \ap B_1 + B_2$, where
    \begin{align*}
    B_1: & = \bigg(\int_0^{\infty} \bigg(\int_0^x  [\overline{u}]^q(t) w(t)\,dt\bigg)^{r / p} [\overline{u}]^q(x) \bigg[\s_p(x,\infty)\bigg]^r w(x)\,dx \bigg)^{1 / r}, \\
    B_2: & = \bigg(\int_0^{\infty} W_*^{r / p}(x) \bigg[\sup_{0 < \tau \le x} \overline{u}(\tau) \s_p (\tau,\infty) \bigg]^r  w(x)\,dx \bigg)^{1 / r}.
    \end{align*}
\end{theorem}

\begin{proof}
Obviously, inequality \eqref{ISI.2}
is satisfied with the best constant $c$ if and only if
\begin{equation}\label{IHI.1.1.1.1}
\left\|S_{\tilde{u}}^* \bigg( \int_0^x h\bigg)\right\|_{q,\tilde{w},\I} \leq c
\,\|h\|_{p,\tilde{v},\I},~ h \in \M^+
\end{equation}
holds, where
$$
\tilde{u} (t) = u \bigg(\frac{1}{t}\bigg), ~ \tilde{w} (t) = w \bigg(\frac{1}{t}\bigg)\frac{1}{t^2}, ~\tilde{v} (t) = v \bigg(\frac{1}{t}\bigg)\bigg(\frac{1}{t^2}\bigg)^{1-p}, ~ t > 0.
$$
Using Theorem \ref{supr.thm.101}, and then applying substitution of variables mentioned above three times, we get the statement.
\end{proof}

\begin{theorem}\label{supr.thm.11}
    Let $0 < p,\,q < \infty$ and let $u \in \W\I \cap C\I$. Assume that $v,\,w \in \W\I$
    be such that $0 < V_*(x) < \infty$ and $0 < W_*(x) < \infty$ for all $x > 0$. Then inequality \eqref{supr.ineq.1} is satisfied with the best constant $c$ if and only if:

    {\rm (i)} $p \le q$, and in this case
    $c \ap A_1 + \| S_u ({\bf 1})\|_{q,w,\I} / \|{\bf 1}\|_{s,v,\I}$, where
    $$
    A_1: = \sup_{x > 0}\bigg( [\overline{u}]^q(x) W_*(x) + \int_0^x  [\overline{u}]^q(t)  w(t)\,dt\bigg)^{1 / q}V^{- 1 / p}(x);
    $$

    {\rm (ii)} $q < p$, and in this case
    $c \ap B_1 + B_2 + \| S_u ({\bf 1})\|_{q,w,\I} / \|{\bf 1}\|_{s,v,\I}$, where
    \begin{align*}
    B_1: & = \bigg(\int_0^{\infty} \bigg(\int_0^x  [\overline{u}]^q(t) w(t)\,dt\bigg)^{r / p} [\overline{u}]^q(x) V^{- r / p}(x) w(x)\,dx \bigg)^{1 / r}, \\
    B_2: & = \bigg(\int_0^{\infty} W_*^{r / p}(x) \bigg[\sup_{0 < \tau \le x} \overline{u}(\tau)V^{- 1 / p}(\tau)\bigg]^{r}  w(x)\,dx \bigg)^{1 / r}.
    \end{align*}
\end{theorem}

\begin{proof}
    It is easy to see that inequality \eqref{supr.ineq.1} holds if and only if
    \begin{equation}\label{supr.ineq.thm.11.eq.1}
    \| S_{u^p} (f)\|_{q/p,w,\I} \le c^p\|f\|_{1,v,\I},~ f \in \M^{\dn}
    \end{equation}
    holds. By Theorem \ref{Reduction.Theorem.thm.3.1}, \eqref{supr.ineq.thm.11.eq.1} holds iff both
    \begin{equation}\label{supr.ineq.thm.11.eq.2}
    \bigg\| S_{u^p} \bigg(\int_x^{\infty} h\bigg) \bigg\|_{q/p,w,\I} \le c^p\|h\|_{1,V,\I},~ h \in \M^+,
    \end{equation}
    and
    \begin{equation}\label{supr.ineq.thm.11.eq.3}
    \| S_u ({\bf 1})\|_{q,w,\I} \le c\|{\bf 1}\|_{s,v,\I}
    \end{equation}
    hold. In order to complete the proof, it remains to apply Theorem \ref{supr.thm.111}.
\end{proof}

Using change of variables $x = 1/t$, we can easily obtain the following "dual" statement.

\begin{theorem}\label{supr.thm.12}
    Let $0 < p,\,q < \infty$ and let $u \in \W\I \cap C\I$. Assume that $v,\,w \in \W\I$
    be such that $0 < V(x) < \infty$ and $0 < W(x) < \infty$ for all $x > 0$. Then \eqref{supr.ineq.2} is satisfied with the best constant $c$ if and only if:

    {\rm (i)} $p \le q$, and in this case
    $c \ap A_1 + \| S_u^* ({\bf 1})\|_{q,w,\I} / \|{\bf 1}\|_{s,v,\I}$, where
    $$
    A_1: = \sup_{x > 0}\bigg( [\underline{u}]^q(x) W(x) + \int_x^{\infty}  [\underline{u}]^q(t)  w(t)\,dt\bigg)^{1 / q}V_*^{- 1 / p}(x);
    $$

    {\rm (ii)} $q < p$, and in this case
    $c \ap B_1 + B_2 + \| S_u^* ({\bf 1})\|_{q,w,\I} / \|{\bf 1}\|_{s,v,\I}$, where
    \begin{align*}
    B_1: & = \bigg(\int_0^{\infty} \bigg(\int_x^{\infty}  [\underline{u}]^q(t) w(t)\,dt\bigg)^{r / p} [\underline{u}]^q(x) V_*^{- r / p}(x) w(x)\,dx \bigg)^{1 / r}, \\
    B_2: & = \bigg(\int_0^{\infty} W^{r / p}(x) \bigg[\sup_{x \le \tau < \infty} \underline{u}(\tau)V_*^{-1 / p}(\tau)\bigg]^{r}  w(x)\,dx \bigg)^{1 / r}.
    \end{align*}
\end{theorem}

\begin{proof}
It is easy to see that inequality \eqref{supr.ineq.2} is satisfied with the best constant $c$ if and only if
\begin{equation}\label{IHI.1.1.1.1.1}
\|S_{\tilde{u}} f \|_{q,\tilde{w},\I} \leq c
\,\|f\|_{p,\tilde{v},\I},~ f \in \M^{\dn}
\end{equation}
holds, where
$$
\tilde{u} (t) = u \bigg(\frac{1}{t}\bigg), ~ \tilde{w} (t) = w \bigg(\frac{1}{t}\bigg)\frac{1}{t^2}, ~\tilde{v} (t) = v \bigg(\frac{1}{t}\bigg)\frac{1}{t^2}, ~ t > 0.
$$
Using Theorem \ref{supr.thm.11}, and then applying substitution of variables mentioned above three times, we get the statement.
\end{proof}

\begin{theorem}\label{supr.thm.23}
    Let $0 < p,\,q < \infty$ and let $u \in \W\I \cap C\I$. Assume that $v,\,w \in \W\I$
    be such that $0 < V_*(x) < \infty$ and $0 < W_*(x) < \infty$ for all $x > 0$.
    Then \eqref{supr.ineq.3} is satisfied with the best constant $c$ if and only if:

    {\rm (i)} $p \le q$, and in this case $c \ap A_1 + \| S_u ({\bf 1}) \|_{q,w,(0,\i)} / \| {\bf 1} \|_{p,v,(0.\i)}$, where
    $$
    A_1: = \sup_{x > 0}\bigg( \bigg[\sup_{0 < \tau \le x} \frac{u(\tau)^p}{V_*(\tau)^2}\bigg]^{q / p} W_*(x) + \int_0^x  \bigg[\sup_{0 < \tau \le t} \frac{u(\tau)^p}{V_*(\tau)^2}\bigg]^{q / p} w(t)\,dt\bigg)^{1 / q}[V_*]^{ 1 / p}(x);
    $$

    {\rm (ii)} $0 < q < p < \infty$,  and in this case $c \ap A_2 + \| S_u ({\bf 1}) \|_{q,w,(0,\i)} / \| {\bf 1} \|_{p,v,(0.\i)}$, where
    \begin{align*}
    B_1: & = \bigg(\int_0^{\infty} \bigg(\int_0^x  \bigg[\sup_{0 < \tau \le t} \frac{u(\tau)^p}{V_*(\tau)^2}\bigg]^{q / p} w(t)\,dt\bigg)^{r / p} [V_*]^{- r / p}(x) \bigg[\sup_{0 < \tau \le x} \frac{u(\tau)^p}{V_*(\tau)^2}\bigg]^{q / p}  w(x)\,dx \bigg)^{1 / r}, \\
    B_2: & = \bigg(\int_0^{\infty} W_*^{r / p}(x) \bigg[\sup_{0 < \tau \le x} \bigg[\sup_{0 < \tau \le t} \frac{u(\tau)^p}{V_*(\tau)^2}\bigg] V_*(\tau)\bigg]^{r / p}  w(x)\,dx \bigg)^{1 / r}.
    \end{align*}
\end{theorem}

\begin{proof}
    It is easy to see that inequality \eqref{supr.ineq.3} holds if and only if
    \begin{equation}\label{supr.ineq.thm.23.eq.1}
    \| S_{u^p} (f)\|_{q/p,w,\I} \le c^p\|f\|_{1,v,\I},~ f \in \M^{\up}
    \end{equation}
    holds. By Theorem  \ref{Reduction.Theorem.thm.3.4} applied to the operator $S_{u^p}$, inequality \eqref{supr.ineq.thm.23.eq.1} is satisfied with the best constant $c$ if and only if
both
\begin{equation*}
\left\|  S_{ u^p / V_*^2} \bigg( \int_{x}^{\infty} h\bigg) \right\|_{q/p, w, \I} \le c \|h\|_{1,1 / V_*,\I}, \, h \in \M^+,
\end{equation*}
and
\begin{equation*}
\| S_u ({\bf 1}) \|_{q,w,(0,\i)} \le c \| {\bf 1} \|_{p,v,(0.\i)}
\end{equation*}
hold. It remains to apply Theorem \ref{supr.thm.111}.
\end{proof}

The following "dual" statement holds true.
\begin{theorem}\label{supr.thm.33}
    Let $0 < p,\,q < \infty$ and let $u \in \W\I \cap C\I$. Assume that $v,\,w \in \W\I$
    be such that $0 < V(x) < \infty$ and $0 < W(x) < \infty$ for all $x > 0$.
    Then \eqref{supr.ineq.4}
    is satisfied with the best constant $c$ if and only if:

    {\rm (i)} $p \le q$, and in this case $c \ap A_1 + \| S_u^* ({\bf 1}) \|_{q,w,(0,\i)} / \| {\bf 1} \|_{p,v,(0.\i)}$, where
    $$
A_1: = \sup_{x > 0}\bigg( \bigg[\sup_{x \le \tau < \infty} \frac{u(\tau)^p}{V(\tau)^2}\bigg]^{q / p} W(x) + \int_x^{\infty}  \bigg[\sup_{t \le \tau < \infty} \frac{u(\tau)^p}{V(\tau)^2}\bigg]^{q / p} w(t)\,dt\bigg)^{1 / q} V^{ 1 / p}(x);
$$

    {\rm (ii)} $q < p$,  and in this case $c \ap B_1 + B_2 + \| S_u^* ({\bf 1}) \|_{q,w,(0,\i)} / \| {\bf 1} \|_{p,v,(0.\i)}$, where
    \begin{align*}
    B_1: & = \bigg(\int_0^{\infty} \bigg(\int_x^{\infty}  \bigg[\sup_{t \le \tau < \infty} \frac{u(\tau)^p}{V(\tau)^2}\bigg]^{q / p} w(t)\,dt\bigg)^{r / p} V^{- r / p}(x) \bigg[\sup_{x \le \tau < \infty} \frac{u(\tau)^p}{V(\tau)^2}\bigg]^{q / p}  w(x)\,dx \bigg)^{1 / r}, \\
    B_2: & = \bigg(\int_0^{\infty} W^{r / p}(x) \bigg[\sup_{x \le \tau < \infty} \bigg[\sup_{t \le \tau < \infty} \frac{u(\tau)^p}{V(\tau)^2}\bigg] V(\tau)\bigg]^{r / p}  w(x)\,dx \bigg)^{1 / r}.
    \end{align*}\end{theorem}

\begin{proof}
Obviously, \eqref{supr.ineq.4}
is satisfied with the best constant $c$ if and only if
\begin{equation*}
\|S_{\tilde{u}} f \|_{q,\tilde{w},\I} \leq c
\,\|f\|_{p,\tilde{v},\I},~ f \in \M^{\up}
\end{equation*}
holds, where
$$
\tilde{u} (t) = u \bigg(\frac{1}{t}\bigg), ~ \tilde{w} (t) = w \bigg(\frac{1}{t}\bigg)\frac{1}{t^2}, ~\tilde{v} (t) = v \bigg(\frac{1}{t}\bigg)\frac{1}{t^2}, ~ t > 0.
$$
Using Theorem \ref{supr.thm.23}, and then applying substitution of variables mentioned above three times, we get the statement.
\end{proof}



\section{Iterated inequalities with supremal operators}\label{s.4}

In this section we characterize inequalities \eqref{ISI.1} and \eqref{ISI.3}.

The following theorem is true.
\begin{theorem}\label{supr.thm.41}
    Let $1 < p < \infty$, $0 < q < \infty$ and let $u \in \W\I \cap C\I$. Assume that $v,\,w \in \W\I$
    be such that
    $$
    0 < \int_0^x v^{1-p^{\prime}}(t)\,dt < \i \qquad \mbox{and} \qquad  0 < W_*(x) < \infty \qq \mbox{for all} \qq x > 0.
    $$
    Recall that
    $$
    \Phi \big[v;p\big](x) = \bigg( \int_0^x
    v^{1-{p}^{\prime}}(t)\,dt\bigg)^{{1} / {(p^{\prime} + 1)}},~ x > 0.
    $$
    Denote by
    $$
    \Phi_1(x) : = \sup_{0 < \tau \le x} u(\tau) \Phi^2[v;p](\tau), ~ x>0.
    $$
    Then \eqref{ISI.1}
    is satisfied with the best constant $c$ if and only if:

    {\rm (i)} $p \le q$, and in this case
    $c \ap A_1 + \| S_{u\Phi^2[v;p]} ({\bf 1})\|_{q,w,\I} / \|{\bf 1}\|_{p,\phi[v;p],\I}$, where
    $$
    A_1: = \sup_{x > 0}\bigg( [\Phi_1]^q(x) W_*(x) + \int_0^x  [\Phi_1]^q(t)  w(t)\,dt\bigg)^{1 / q}\Phi[v;p]^{- 1 / p}(x);
    $$

    {\rm (ii)} $q < p$, and in this case
    $c \ap B_1 + B_2 + \| S_{u\Phi^2[v;p]} ({\bf 1})\|_{q,w,\I} / \|{\bf 1}\|_{p,\phi[v;p],\I}$, where
    \begin{align*}
    B_1: & = \bigg(\int_0^{\infty} \bigg(\int_0^x  [\Phi_1]^q(t) w(t)\,dt\bigg)^{r / p} [\Phi_1]^q(x) \Phi[v;p]^{- r / p}(x) w(x)\,dx \bigg)^{1 / r}, \\
    B_2: & = \bigg(\int_0^{\infty} W_*^{r / p}(x) \bigg[\sup_{0 < \tau \le x} \Phi_1(\tau)\Phi[v;p]^{-1 / p}(\tau)\bigg]^{r}  w(x)\,dx \bigg)^{1 / r}.
    \end{align*}
\end{theorem}

\begin{proof}
By Theorem \ref{RT.thm.main.3} applied to the
operator $S_{u}$, inequality \eqref{ISI.1} with the best constant $c$ holds if and only if the inequality
\begin{equation}\label{RT.SC.thm.41.eq.2}
\big\| S_{u \Phi^2[v;p]} (f)\big\|_{q, w, \I} \le C \,\|
f \|_{p, \phi[v;p],\I}, \, f \in \M^{\dn}
\end{equation}
holds. Moreover, $c \ap C$. Now the statement follows by Theorem \ref{supr.thm.11}.
\end{proof}

\begin{theorem}\label{supr.thm.61}
    Let $0 < q < \infty$ and let $u \in \W\I \cap C\I$. Assume that $v,\,w \in \W\I$
    be such that $0 < V(x) < \infty$ and $0 < W_*(x) < \infty$ for all $x > 0$. Denote by
    $$
    V_1(x) : = \sup_{0 < \tau \le x} u(\tau) V^2(\tau), ~ x>0.
    $$
    Then
    \begin{equation}\label{RT.SC.thm.51.eq.1}
    \left\|S_{u} \bigg( \int_0^x h\bigg)\right\|_{q,w,\I} \leq c
    \,\|h\|_{1,V^{-1},\I},
    \end{equation}
    is satisfied with the best constant $c$ if and only if:

    {\rm (i)} $p \le q$, and in this case
    $c \ap A_1 + \| S_{uV^2} ({\bf 1})\|_{q,w,\I} / \|{\bf 1}\|_{p,v,\I}$, where
    $$
    A_1: = \sup_{x > 0}\bigg( [V_1]^q(x) \int_x^{\infty} w(t)\,dt + \int_0^x  [V_1]^q(t)  w(t)\,dt\bigg)^{1 / q}V^{- 1 / p}(x);
    $$

    {\rm (ii)} $q < p$, and in this case
    $c \ap B_1 + B_2 + \| S_{uV^2} ({\bf 1})\|_{q,w,\I} / \|{\bf 1}\|_{p,v,\I}$, where
    \begin{align*}
    B_1: & = \bigg(\int_0^{\infty} \bigg(\int_0^x  [V_1]^q(t) w(t)\,dt\bigg)^{r / p} [V_1]^q(x) V^{- r / p}(x) w(x)\,dx \bigg)^{1 / r}, \\
    B_2: & = \bigg(\int_0^{\infty} \bigg(\int_x^{\infty}  w(t)\,dt\bigg)^{r / p} \bigg[\sup_{0 < \tau \le x} [V_1](\tau)V^{-1 / p}(\tau)\bigg]^{r}  w(x)\,dx \bigg)^{1 / r}.
    \end{align*}
    \end{theorem}

\begin{proof}
By Theorem \ref{RT.thm.main.8.0} applied to the
operator $S_{u}$, inequality \eqref{RT.SC.thm.51.eq.1}
with the best constant $c_{51}$ holds if and
only if the inequality
\begin{equation}\label{RT.SC.thm.51.eq.2}
\big\| S_{u V^2} (f)\big\|_{q, w, \I} \le C \,\| f \|_{1, v,\I}, \, f \in \M^{\dn}
\end{equation}
holds. Moreover, $c \ap C$. Now the statement follows by Theorem  \ref{supr.thm.11}.
\end{proof}

The following "dual" statements also hold true.
\begin{theorem}\label{supr.thm.71}
    Let $1 < p < \infty$, $0 < q < \infty$ and let $u \in \W\I \cap C\I$. Assume that $v,\,w \in \W\I$
    be such that
    $$
    0 < \int_x^{\infty} v^{1-p^{\prime}}(t)\,dt < \i \qquad \mbox{and} \qquad  0 < W(x) < \infty \qq \mbox{for all} \qq x > 0.
    $$
    Recall that
    $$
    \Psi \big[v;p\big](x) = \bigg( \int_x^{\infty}
    v^{1-{p}^{\prime}}(t)\,dt\bigg)^{{1} / {(p^{\prime}+ 1)}}, ~ x > 0.
    $$
    Denote by
    $$
    \Psi_1(x) : = \sup_{x \le \tau < \infty} u(\tau) \Psi^2[v;p](\tau), ~ x>0.
    $$
    Then \eqref{ISI.3} is satisfied with the best constant $c$ if and only if:

    {\rm (i)} $p \le q$, and in this case
    $c \ap A_1 + \| S_{u\Psi^2[v;p]} ({\bf 1})\|_{q,w,\I} / \|{\bf 1}\|_{p,\psi[v;p],\I}$, where
    $$
    A_1: = \sup_{x > 0}\bigg( [\Psi_1]^q(x) W(x) + \int_x^{\infty}  [\Psi_1]^q(t)  w(t)\,dt\bigg)^{1 / q}\Psi[v;p]^{- 1 / p}(x);
    $$

    {\rm (ii)} $q < p$, and in this case
    $c \ap B_1 + B_2 + \| S_{u\Psi^2[v;p]} ({\bf 1})\|_{q,w,\I} / \|{\bf 1}\|_{p,\psi[v;p],\I}$, where
    \begin{align*}
    B_1: & = \bigg(\int_0^{\infty} \bigg(\int_x^{\infty}  [\Psi_1]^q(t) w(t)\,dt\bigg)^{r / p} [\Psi_1]^q(x) \Psi[v;p]^{- r / p}(x) w(x)\,dx \bigg)^{1 / r}, \\
    B_2: & = \bigg(\int_0^{\infty} W^{r / p}(x) \bigg[\sup_{x \le \tau < \infty} \Psi_1(\tau)\Psi[v;p]^{-1 / p}(\tau)\bigg]^{r}  w(x)\,dx \bigg)^{1 / r}.
    \end{align*}
\end{theorem}

\begin{proof}
Obviously, \eqref{ISI.3}
is satisfied with the best constant $c$ if and only if
\begin{equation*}
\bigg\|S_{\tilde{u}} \bigg( \int_0^x h\bigg) \bigg\|_{q,\tilde{w},\I} \leq c
\,\|h\|_{p,\tilde{v},\I},~ h \in \M^+
\end{equation*}
holds, where
$$
\tilde{u} (t) = u \bigg(\frac{1}{t}\bigg), ~ \tilde{w} (t) = w \bigg(\frac{1}{t}\bigg)\frac{1}{t^2}, ~\tilde{v} (t) = v \bigg(\frac{1}{t}\bigg)\frac{1}{t^2}, ~ t > 0.
$$
Using Theorem \ref{supr.thm.41}, and then applying substitution of variables mentioned above three times, we get the statement.
\end{proof}

\begin{theorem}\label{supr.thm.81}
    Let $0 < q < \infty$ and let $u \in \W\I \cap C\I$. Assume that $v,\,w \in \W\I$
    be such that $0 < V_*(x) < \infty$ and $0 < W_*(x) < \infty$ for all $x > 0$. Denote by
    $$
    V_1^*(x) : = \sup_{x \le \tau < \infty} u(\tau) V_*^2(\tau), ~ x>0.
    $$
    Then
    \begin{equation}\label{RT.SC.thm.81.eq.1}
    \left\|S_{u}^* \bigg( \int_x^{\infty} h\bigg)\right\|_{q,w,\I} \leq c
    \,\|h\|_{1,V_*^{-1},\I},
    \end{equation}
    is satisfied with the best constant $c$ if and only if:

    {\rm (i)} $p \le q$, and in this case
    $c \ap A_1 + \big\| S_{uV_*^2}^* ({\bf 1})\big\|_{q,w,\I} / \|{\bf 1}\|_{p,v,\I}$, where
    $$
    A_1: = \sup_{x > 0}\bigg( [V_1^*]^q(x) W_*(x) + \int_0^x  [V_1^*]^q(t)  w(t)\,dt\bigg)^{1 / q}V^{- 1 / p}(x);
    $$

    {\rm (ii)} $q < p$, and in this case
    $c \ap B_1 + B_2 + \big\| S_{uV_*^2}^* ({\bf 1})\big\|_{q,w,\I} / \|{\bf 1}\|_{p,v,\I}$, where
    \begin{align*}
    B_1: & = \bigg(\int_0^{\infty} \bigg(\int_x^{\infty}  [V_1^*]^q(t) w(t)\,dt\bigg)^{r / p} [V_1^*]^q(x) V_*^{- r / p}(x) w(x)\,dx \bigg)^{1 / r}, \\
    B_2: & = \bigg(\int_0^{\infty} W_*^{r / p}(x) \bigg[\sup_{x \le \tau < \infty} [V_1^*](\tau)V^{-1 / p}(\tau)\bigg]^{r}  w(x)\,dx \bigg)^{1 / r}.
    \end{align*}
\end{theorem}

\begin{proof}
    By change of variables $x = 1 / t$, it is easy to see that inequality \eqref{RT.SC.thm.81.eq.1} holds if and only if
    \begin{equation}\label{IHI.1.1.1}
    \left\|S_{\tilde{u}} \bigg( \int_0^x h\bigg)\right\|_{q,\tilde{w},\I} \leq c
    \,\|h\|_{1,\tilde{V}^{-1},\I},~h\in \M^+
    \end{equation}
    holds, where
    $$
    \tilde{u} (t) = u \bigg(\frac{1}{t}\bigg), ~ \tilde{w} (t) = w \bigg(\frac{1}{t}\bigg)\frac{1}{t^2}, ~\tilde{V} (t) = \int_0^t v \bigg(\frac{1}{y}\bigg)\frac{1}{y^2}\,dy, ~ t > 0.
    $$
    Applying Theorem \ref{supr.thm.61}, and then using substitution of variables mentioned above three times, we get the statement.
\end{proof}


\section{Hardy-operator involving suprema - $T_{u,b}$}\label{s.5}

In this section we give complete characterization of inequality \eqref{Tub.thm.1.eq.1}.

\subsection{The case $1 \le p < \infty$}

The following theorem is true.

\begin{theorem}\label{Tub.thm.1}
Let $0 < q \le \infty$, $1 \le p < \infty$ and let $u \in \W\I \cap C\I$. Assume that $b,\,v,\,w \in \W\I$
be such that
$$
0 < B(t) < \infty, ~ 0 < V(x) < \i ~ \mbox{and} ~  0 < W(x) < \infty ~ \mbox{for all} ~ x > 0.
$$
Then inequality \eqref{Tub.thm.1.eq.1}
is satisfied with the best constant $c$ if and only if:

    {\rm (i)} $1 < p \le q$, and in this case $c \ap A_1 + A_2$, where
    \begin{align*}
    A_1: & = \sup_{x > 0}\bigg( \bigg[\sup_{x \le \tau < \infty} \frac{u(\tau)}{B(\tau)}\bigg]^q  W(x) + \int_x^{\infty}  \bigg[\sup_{t \le \tau < \infty} \frac{u(\tau)}{B(\tau)}\bigg]^q w(t)\,dt\bigg)^{1 / q}\bigg(\int_0^x \bigg(\frac{B(y)}{V(y)}\bigg)^{p'}v(y)\,dy\bigg)^{1 / p'}, \\
    A_2: &  = \sup_{x > 0}\bigg( \bigg[\sup_{x \le \tau < \infty} \frac{u(\tau)}{V^2(\tau)}\bigg]^q  W(x) + \int_x^{\infty}  \bigg[\sup_{t \le \tau < \infty} \frac{u(\tau)}{V^2(\tau)}\bigg]^q w(t)\,dt\bigg)^{1 / q}\bigg(\int_0^x V^{p'}(y)v(y)\,dy\bigg)^{1 / p'};
    \end{align*}

    {\rm (ii)} $1 = p \le q$, and in this case $c \ap A_1 + A_2$, where
    \begin{align*}
    A_1: & = \sup_{x > 0}\bigg( \bigg[\sup_{x \le \tau < \infty} \frac{u(\tau)}{B(\tau)}\bigg]^q  W(x) + \int_x^{\infty}  \bigg[\sup_{t \le \tau < \infty} \frac{u(\tau)}{B(\tau)}\bigg]^q w(t)\,dt\bigg)^{1 / q}\bigg(\sup_{0 < y \le x} \frac{B(y)}{V(y)}\bigg), \\
    A_2: &  = \sup_{x > 0}\bigg( \bigg[\sup_{x \le \tau < \infty} \frac{u(\tau)}{V^2(\tau)}\bigg]^q  W(x) + \int_x^{\infty}  \bigg[\sup_{t \le \tau < \infty} \frac{u(\tau)}{V^2(\tau)}\bigg]^q w(t)\,dt\bigg)^{1 / q}V(x);
    \end{align*}

    {\rm (iii)} $1 < p$ and $q < p$,  and in this case $c \ap B_1 + B_2 + B_3 + B_4$, where
    \begin{align*}
    B_1: & = \bigg(\int_0^{\infty} \bigg(\int_x^{\infty}  \bigg[\sup_{t \le \tau < \infty} \frac{u(\tau)}{B(\tau)}\bigg]^q  w(t)\,dt\bigg)^{r / p} \bigg[\sup_{x \le \tau < \infty} \frac{u(\tau)}{B(\tau)}\bigg]^q \bigg(\int_0^x \bigg(\frac{B(y)}{V(y)}\bigg)^{p'}v(y)\,dy\bigg)^{r / p'} w(x)\,dx \bigg)^{1/r}, \\
    B_2: & = \bigg(\int_0^{\infty} W^{r / p}(x) \bigg[\sup_{x \le \tau < \infty} \bigg[\sup_{\tau \le y < \infty} \frac{u(y)}{B(y)}\bigg] \bigg(\int_0^x \bigg(\frac{B(y)}{V(y)}\bigg)^{p'}v(y)\,dy\bigg)^{1 / p'} \bigg]^r  w(x)\,dx \bigg)^{1/r},\\
    B_3: & = \bigg(\int_0^{\infty} \bigg(\int_x^{\infty}  \bigg[\sup_{t \le \tau < \infty} \frac{u(\tau)}{V^2(\tau)}\bigg]^q  w(t)\,dt\bigg)^{r / p} \bigg[\sup_{x \le \tau < \infty} \frac{u(\tau)}{V^2(\tau)}\bigg]^q \bigg(\int_0^x V^{p'}(y)v(y)\,dy\bigg)^{r / p'} w(x)\,dx \bigg)^{1/r}, \\
    B_4: & = \bigg(\int_0^{\infty} W^{r / p}(x) \bigg[\sup_{x \le \tau < \infty} \bigg[\sup_{\tau \le y < \infty} \frac{u(y)}{V^2(y)}\bigg] \bigg(\int_0^x V^{p'}(y)v(y)\,dy\bigg)^{1 / p'} \bigg]^r  w(x)\,dx \bigg)^{1/r}.
    \end{align*}

    {\rm (iv)} $q < 1 = p$,  and in this case $c \ap B_1 + B_2 + B_3 + B_4$, where
    \begin{align*}
    B_1: & = \bigg(\int_0^{\infty} \bigg(\int_x^{\infty}  \bigg[\sup_{t \le \tau < \infty} \frac{u(\tau)}{B(\tau)}\bigg]^q  w(t)\,dt\bigg)^{r / p} \bigg[\sup_{x \le \tau < \infty} \frac{u(\tau)}{B(\tau)}\bigg]^q \bigg(\sup_{0 < y \le x} \frac{B(y)}{V(y)}\bigg)^{r} w(x)\,dx \bigg)^{1/r}, \\
    B_2: & = \bigg(\int_0^{\infty} W^{r / p}(x) \bigg[\sup_{x \le \tau < \infty} \bigg[\sup_{\tau \le y < \infty} \frac{u(y)}{B(y)}\bigg] \bigg(\sup_{0 < y \le x} \frac{B(y)}{V(y)}\bigg) \bigg]^r  w(x)\,dx \bigg)^{1/r},\\
    B_3: & = \bigg(\int_0^{\infty} \bigg(\int_x^{\infty}  \bigg[\sup_{t \le \tau < \infty} \frac{u(\tau)}{V^2(\tau)}\bigg]^q  w(t)\,dt\bigg)^{r / p} \bigg[\sup_{x \le \tau < \infty} \frac{u(\tau)}{V^2(\tau)}\bigg]^q V^r(x) w(x)\,dx \bigg)^{1/r}, \\
    B_4: & = \bigg(\int_0^{\infty} W^{r / p}(x) \bigg[\sup_{x \le \tau < \infty} \bigg[\sup_{\tau \le y < \infty} \frac{u(y)}{V^2(y)}\bigg] V(x) \bigg]^r  w(x)\,dx \bigg)^{1/r}.
    \end{align*}
\end{theorem}

\begin{proof}
By Theorem \ref{Reduction.Theorem.thm.3.1}, \eqref{Tub.thm.1.eq.1} holds iff both
    \begin{equation}\label{Tub.thm.1.eq.5}
    \bigg\| T_{u,b}\bigg( \int_x^{\i} h \bigg)\bigg\|_{q,w,\I} \le c \| h
    \|_{p, V^p v^{1-p},\I}, \qq h \in \M^+(0,\i).
    \end{equation}
    and
    \begin{equation}\label{Tub.thm.1.eq.6}
    \|T_{u,b} {\bf 1}\|_{q,w,\I} \le c \|{\bf 1}\|_{p,v,\I}
    \end{equation}
    hold.

Note that
\begin{align*}
T_{u,b}\bigg( \int_{t}^{\i} h \bigg) (x) & = \sup_{x \le \tau < \infty} \frac{u(\tau)}{B(\tau)} \int_0^{\tau} \bigg( \int_s^{\infty} h(y)\,dy\bigg) b(s)\,ds \\
& \ap \sup_{x \le \tau < \infty} \frac{u(\tau)}{B(\tau)} \int_0^{\tau} h(y)B(y)\,dy + \sup_{x \le \tau < \infty} u(\tau) \int_{\tau}^{\infty} h(s)\,ds \\
& = S_{u/B}^* \bigg(\int_0^{\tau} hB\bigg) + S_u^* \bigg(\int_{\tau}^{\infty} h\bigg).
\end{align*}

Hence, inequality \eqref{Tub.thm.1.eq.1}
holds iff inequalities
\begin{equation}\label{Tub.thm.1.eq.2}
\bigg\|S_{u/B}^* \bigg(\int_0^{\tau} h\bigg)\bigg\|_{q,w,\I} \le c \| h
\|_{p, B^{-p}V^p v^{1-p},\I}, \qq h \in \M^+(0,\i),
\end{equation}
\begin{equation}\label{Tub.thm.1.eq.3}
\bigg\| S_u^* \bigg(\int_{\tau}^{\infty} h\bigg) \bigg\|_{q,w,\I} \le c \| h
\|_{p, V^p v^{1-p},\I}, \qq h \in \M^+(0,\i).
\end{equation}
and \eqref{Tub.thm.1.eq.6} hold.

Again by Theorem \ref{Reduction.Theorem.thm.3.1}, \eqref{Tub.thm.1.eq.3} with \eqref{Tub.thm.1.eq.6} is equivalent to
\begin{equation}\label{Tub.thm.1.eq.4}
\|S_u^* f \|_{q,w,\I} \le c \| f \|_{p,v,\I}, \qq f \in \M^{\dn}(0,\i).
\end{equation}

Now by Theorem \ref{Reduction.Theorem.thm.3.2}, \eqref{Tub.thm.1.eq.4} is equivalent to
\begin{equation}\label{Tub.thm.1.eq.5.0}
    \bigg\| S_{u / V^2}^* \bigg( \int_0^x h \bigg)\bigg\|_{q,w,\I}
    \le c \| h \|_{p, V^{-p}v^{1-p},\I}, \qq h \in \M^+(0,\i).
\end{equation}

Consequently, \eqref{Tub.thm.1.eq.1}
holds iff inequalities \eqref{Tub.thm.1.eq.2} and \eqref{Tub.thm.1.eq.5.0} hold.

    {\rm (i)} $p \le q$. By Theorem \ref{supr.thm.101}, \eqref{Tub.thm.1.eq.2} and \eqref{Tub.thm.1.eq.5.0} hold if and only if both
    $A_1 < \infty$ and $A_2 < \infty$
    hold, respectively.

     {\rm (ii)} $q < p$. By Theorem \ref{supr.thm.101}, \eqref{Tub.thm.1.eq.2} and \eqref{Tub.thm.1.eq.5.0} hold if and only if
     $B_i < \infty$, $i=1,2,3,4$
     hold, respectively.
\end{proof}

\subsection{The case $0< p < 1$}

We start with a simple observation. If $0 < p \le 1$ and $t \in \I$, then
\begin{equation}\label{RT.SO.eq.1.1}
\sup_{0 < \tau \le t} f (\tau) B(\tau) \le \int_0 ^t f(y)b(y)\,dy \ls \bigg( \int_0^t f(y)^p B(y)^{p-1} b(y)\,dy\bigg) ^{{1} / {p}}, ~ f \in \M^{\dn}.
\end{equation}
Since $f$ is non-increasing, the first inequality in \eqref{RT.SO.eq.1.1} is obvious. The second one follows, for instance, from the fact that (see, for instance, \cite[Theorem 3.2]{carsor1993}, cf. also \cite{ss})
$$
\sup_{f\in \M^{\dn}:\,f \not \sim 0} \frac{\int_0^{\infty} f(x)g(x)\,dx}{\bigg(\int_0^{\infty} f(x)^p v(x)\,dx\bigg)^{1/p}} \ap \sup_{t > 0} \bigg( \int_0^t g(x)\,dx \bigg(\int_0^t v(x)\,dx\bigg)^{- 1/ p}\bigg).
$$

Our first aim is to prove a reduction theorem for the operator $T_{u,b}$. We first note that, using the monotonicity of $\int_0^t fb$ and interchanging the suprema, we get
$$
(T_{u,b} g)(t) = \sup_{t \le \tau < \infty} \frac{u(\tau)}{B(\tau)} \int_0^{\tau} g(y)b(y)\,dy = \sup_{t \le \tau < \infty} \bigg(\sup_{\tau \le x < \infty}\frac{u(x)}{B(x)}\bigg) \int_0^{\tau} g(y)b(y)\,dy,\qquad t \in \I.
$$
As a consequence, we can safely assume that ${u(x)} / {B(x)}$ is non-increasing on $\I$, since otherwise we would just replace ${u(x)} / {B(x)}$ by $\sup_{\tau \le x < \infty}{u(x)} / {B(x)}$.
\begin{theorem}\label{RT.SO.thm.2}
    Let $0 < p \le 1$, $0 < q < \infty$. Assume that $u \in \W\I \cap C\I$ and $b,\,v,\,w \in \W\I$ be such that $0 < B(t) < \infty$ for all $x > 0$. Then the following three statements are equivalent:

    \begin{align}
    \bigg( \int_0^{\infty} \bigg( \sup_{t \le \tau < \infty} \frac{u(\tau)}{B(\tau)}\int_0^{\tau} f(y) b(y) \,dy\bigg)^q w(t)\,dt\bigg)^{1 / q} & \ls \bigg( \int_0^{\infty} f(\tau)^p v(\tau)\,d\tau\bigg)^{1 / p}, ~ f \in \M^{\dn}; \label{RT.SO.thm.2.eq.1} \\
    \bigg( \int_0^{\infty} \bigg( \sup_{t \le \tau < \infty} \bigg(\frac{u(\tau)}{B(\tau)}\bigg)^p \int_0^{\tau} f(y) B(y)^{p-1}b(y) \,dy\bigg)^{q / p} w(t)\,dt\bigg)^{1 / q} & \ls \bigg( \int_0^{\infty} f(\tau) v(\tau)\,d\tau\bigg)^{1 / p}, ~ f \in \M^{\dn}; \label{RT.SO.thm.2.eq.2} \\
    \bigg( \int_0^{\infty} \bigg( \sup_{t \le \tau < \infty} \frac{u(\tau)}{B(\tau)} \sup_{0 < y \le \tau} f(y) B(y) \bigg)^q w(t)\,dt\bigg)^{1 / q} & \ls \bigg( \int_0^{\infty} f(\tau)^p v(\tau)\,d\tau\bigg)^{1 / p}, ~ f \in \M^{\dn}. \label{RT.SO.thm.2.eq.3}
    \end{align}
\end{theorem}

\begin{proof}
Again, in view of \eqref{RT.SO.eq.1.1}, the implications $\eqref{RT.SO.thm.2.eq.2} \Rightarrow \eqref{RT.SO.thm.2.eq.1} \Rightarrow \eqref{RT.SO.thm.2.eq.3}$ are obvious, and it just remains to show that \eqref{RT.SO.thm.2.eq.3} implies \eqref{RT.SO.thm.2.eq.2}.

Suppose, thus, that \eqref{RT.SO.thm.2.eq.3} holds. Since $u(x) / B(x)$ is non-increasing, we have
\begin{align*}
\sup_{t \le \tau < \infty} \frac{u(\tau)}{B(\tau)} \sup_{0 < y \le \tau} f(y) B(y)  & = \max \bigg\{ \sup_{t \le \tau < \infty} \frac{u(\tau)}{B(\tau)} \sup_{0 < y \le t} f(y) B(y),\,\sup_{t \le \tau < \infty} \frac{u(\tau)}{B(\tau)} \sup_{t < y \le \tau} f(y) B(y)  \bigg\} \\
& = \max \bigg\{ \frac{u(t)}{B(t)} \sup_{0 < y \le t} f(y) B(y),\,\sup_{t \le y < \infty} f(y) B(y) \sup_{y \le \tau < \infty} \frac{u(\tau)}{B(\tau)} \bigg\} \\
& = \max \bigg\{ \frac{u(t)}{B(t)} \sup_{0 < y \le t} f(y) B(y),\,\sup_{t \le y < \infty} f(y) u(y) \bigg\}.
\end{align*}

Hence, \eqref{RT.SO.thm.2.eq.3} breaks down into the following two inequalities:
\begin{align}
    \bigg( \int_0^{\infty} \bigg( \sup_{0 < y \le t} f(y) B(y) \bigg)^q w(t) \bigg(\frac{u(t)}{B(t)}\bigg)^q\,dt\bigg)^{1 / q} & \ls \bigg( \int_0^{\infty} f(\tau)^p v(\tau)\,d\tau\bigg)^{1 / p}, ~ f \in \M^{\dn}, \label{RT.SO.thm.2.eq.4} \\
    \bigg( \int_0^{\infty} \bigg( \sup_{t \le y < \infty} f(y) u(y) \bigg)^q w(t) \,dt\bigg)^{1 / q} & \ls \bigg( \int_0^{\infty} f(\tau)^p v(\tau)\,d\tau\bigg)^{1 / p}, ~ f \in \M^{\dn}. \label{RT.SO.thm.2.eq.5}
\end{align}

Obviously, \eqref{RT.SO.thm.2.eq.4} and \eqref{RT.SO.thm.2.eq.5} are equivalent to
\begin{align}
\bigg( \int_0^{\infty} \bigg( \sup_{0 < y \le t} f(y) B(y)^p \bigg)^{q / p} w(t) \bigg(\frac{u(t)}{B(t)}\bigg)^q\,dt\bigg)^{p / q} & \ls \int_0^{\infty} f(\tau) v(\tau)\,d\tau, ~ f \in \M^{\dn}, \label{RT.SO.thm.2.eq.4.1} \\
\bigg( \int_0^{\infty} \bigg( \sup_{t \le y < \infty} f(y) u(y)^p \bigg)^{q / p} w(t) \,dt\bigg)^{p / q} & \ls \int_0^{\infty} f(\tau) v(\tau)\,d\tau, ~ f \in \M^{\dn}. \label{RT.SO.thm.2.eq.5.1}
\end{align}

{\rm (i)} Let $p \le q$. By Theorem \ref{supr.thm.11}, \eqref{RT.SO.thm.2.eq.4.1} holds iff both
\begin{equation}\label{RT.SO.thm.2.eq.9}
\sup_{x > 0} \bigg(\int_0^x u(t)^q w(t)\,dt + B(x)^q\int_x^{\infty} \bigg[ \frac{u(t)}{B(t)}\bigg]^q w(t)\,dt \bigg)^{1 / q}V^{- 1 / p}(x) < \infty
\end{equation}
and
\begin{equation}\label{RT.SO.thm.2.eq.10}
\bigg( \int_0^{\infty} u(t)^q  w(t)\,dt \bigg)^{1 / q} \ls \bigg( \int_0^{\infty} v(\tau)\,d\tau\bigg)^{1 / p}
\end{equation}
hold.

By Theorem \ref{supr.thm.33}, \eqref{RT.SO.thm.2.eq.5.1} holds iff both
\begin{equation}\label{RT.SO.thm.2.eq.12}
\sup_{x > 0}\bigg( \bigg[ \sup_{x \le y < \infty} \frac{u(y)^p}{V^2(y)}\bigg]^{q / p}  \int_0^x w(t)\,dt + \int_x^{\infty}  \bigg[ \sup_{t \le y < \infty} \frac{u(y)^p}{V^2(y)}\bigg]^{q / p} w(t)\,dt\bigg)^{1 / q} V^{1 / p}(x) < \infty
\end{equation}
and
\begin{equation}\label{RT.SO.thm.2.eq.13}
\bigg( \int_0^{\infty} \bigg( \sup_{t \le \tau < \infty} u(\tau)^p\bigg)^{q/p}  w(t)\,dt \bigg)^{1 / q} \ls \bigg( \int_0^{\infty} v(\tau)\,d\tau\bigg)^{1 / p}
\end{equation}
hold.

On the other hand, by Theorem \ref{Tub.thm.1}, \eqref{RT.SO.thm.2.eq.2} holds iff inequalities
\begin{align}
    \sup_{x > 0} \bigg(\bigg[ \frac{u(x)}{B(x)}\bigg]^q \int_0^x  w(t)\,dt + \int_x^{\infty} \bigg[ \frac{u(t)}{B(t)}\bigg]^q w(t)\,dt \bigg)^{1 / q}\sup_{0 < \tau \le x}\frac{B(\tau)}{V^{ 1 / p}(\tau)} & < \infty \label{RT.SO.thm.2.eq.14} \\
    \sup_{x > 0} \bigg(\bigg[\sup_{x \le \tau < \infty} \frac{u(\tau)^p}{V^2(\tau)}\bigg]^{q / p} \int_0^x  w(t)\,dt + \int_x^{\infty} \bigg[ \sup_{t \le \tau < \infty}\frac{u(\tau)^p}{V^2(\tau)}\bigg]^{q / p} w(t)\,dt \bigg)^{1 / q} V^{ 1 / p}(x) & < \infty \label{RT.SO.thm.2.eq.14.1}
\end{align}
hold.

We will thus be done if we can show that  \eqref{RT.SO.thm.2.eq.9} together with \eqref{RT.SO.thm.2.eq.12} imply \eqref{RT.SO.thm.2.eq.14}. The latter can be proved as follows:

Since
$$
\sup_{x > 0} \bigg(\int_x^{\infty} \bigg[ \frac{u(t)}{B(t)}\bigg]^q w(t)\,dt \bigg)^{1 / q}\sup_{0 < \tau \le x}\frac{B(\tau)}{V^{ 1 / p}(\tau)} = \sup_{x > 0} \bigg(\int_x^{\infty} \bigg[ \frac{u(t)}{B(t)}\bigg]^q w(t)\,dt \bigg)^{1 / q}\frac{B(x)}{V^{ 1 / p}(x)},
$$
it remains to show that
\begin{align*}
\sup_{x > 0} \frac{u(x)}{B(x)} \bigg(\int_0^x  w(t)\,dt \bigg)^{1 / q}\sup_{0 < \tau \le x}\frac{B(\tau)}{V^{ 1 / p}(\tau)} & \\
& \hspace{-3cm} \ls \sup_{x > 0} \frac{u(x)}{V^{1 / p}(x)} \bigg(\int_0^x  w(t)\,dt \bigg)^{1 / q} + \sup_{x > 0} \frac{B(x)}{V^{1 / p}(x)} \bigg(\int_x^{\infty} \bigg[ \frac{u(t)}{B(t)}\bigg]^q w(t)\,dt \bigg)^{1 / q}.
\end{align*}
Interchanging the suprema, using the monotonicity of $u / B$, we get that
\begin{align*}
\sup_{x > 0} \frac{u(x)}{B(x)} \bigg(\int_0^x  w(t)\,dt \bigg)^{1 / q}\sup_{0 < \tau \le x}\frac{B(\tau)}{V^{ 1 / p}(\tau)} & \\
& \hspace{-4cm} = \sup_{\tau > 0} \frac{B(\tau)}{V^{ 1 / p}(\tau)}\sup_{\tau \le x < \infty} \frac{u(x)}{B(x)} \bigg(\int_0^x  w(t)\,dt \bigg)^{1 / q} \\
& \hspace{-4cm} \ls \sup_{\tau > 0} \frac{B(\tau)}{V^{ 1 / p}(\tau)}\bigg(\sup_{\tau \le x < \infty} \frac{u(x)}{B(x)} \bigg)\bigg(\int_0^{\tau}  w(t)\,dt \bigg)^{1 / q} + \sup_{\tau > 0} \frac{B(\tau)}{V^{ 1 / p}(\tau)}\sup_{\tau \le x < \infty} \frac{u(x)}{B(x)} \bigg(\int_{\tau}^x  w(t)\,dt \bigg)^{1 / q} \\
& \hspace{-4cm} \ls \sup_{\tau > 0} \frac{u(\tau)}{V^{ 1 / p}(\tau)}\bigg(\int_0^{\tau}  w(t)\,dt \bigg)^{1 / q} + \sup_{\tau > 0} \frac{B(\tau)}{V^{ 1 / p}(\tau)}\sup_{\tau \le x < \infty} \bigg(\int_{\tau}^x \bigg[ \frac{u(t)}{B(t)}\bigg]^q w(t)\,dt \bigg)^{1 / q} \\
& \hspace{-4cm} = \sup_{\tau > 0} \frac{u(\tau)}{V^{ 1 / p}(\tau)}\bigg(\int_0^{\tau}  w(t)\,dt \bigg)^{1 / q} + \sup_{\tau > 0} \frac{B(\tau)}{V^{ 1 / p}(\tau)} \bigg(\int_{\tau}^{\infty} \bigg[ \frac{u(t)}{B(t)}\bigg]^q w(t)\,dt \bigg)^{1 / q}.
\end{align*}

{\rm (ii)} Let $q < p$. By Theorem \ref{supr.thm.11}, \eqref{RT.SO.thm.2.eq.4.1} holds iff both
    \begin{align*}
    \int_0^{\infty} \bigg(\int_0^x  u(t)^q  w(t)\,dt\bigg)^{r / p} u(x)^q V^{- r / p}(x) w(x)\,dx & < \infty, \\
    \int_0^{\infty} \bigg(\int_x^{\infty}  \bigg(\frac{u(t)}{B(t)}\bigg)^q w(t)\,dt\bigg)^{r / p} \bigg[\sup_{0 < \tau \le x} \frac{B(\tau)}{V^{1 / p}(\tau)}\bigg]^{r}w(x) \bigg(\frac{u(x)}{B(x)}\bigg)^q\,dx & < \infty.
    \end{align*}

By Theorem \ref{supr.thm.33}, \eqref{RT.SO.thm.2.eq.5.1} holds iff inequalities
    \begin{align*}
    \int_0^{\infty} \bigg(\int_x^{\infty}  \bigg[ \sup_{t \le y < \infty} \frac{u(y)^p}{V^2(y)}\bigg]^{q / p}  w(t)\,dt\bigg)^{r / p} \bigg[ \sup_{x \le y < \infty} \frac{u(y)^p}{V^2(y)}\bigg]^{q / p} V^{r / p}(x) w(x)\,dx & < \infty, \\
    \int_0^{\infty} \bigg( \int_0^x w(t)\,dt \bigg)^{r / p} \bigg(\sup_{x \le \tau < \infty} \bigg[ \sup_{\tau \le y < \infty} \frac{u(y)^p}{V^2(y)}\bigg] V(\tau) \bigg)^{r / p}  w(x)\,dx & < \infty.
    \end{align*}

On the other hand, by Theorem \ref{Tub.thm.1}, \eqref{RT.SO.thm.2.eq.2} holds iff
\begin{align*}
\int_0^{\infty} \bigg( \int_x^{\infty} \bigg[\frac{u(t)}{B(t)}\bigg]^q w(t)\,dt\bigg)^{r / p} \bigg[\sup_{0 < t \le \tau} \frac{B(t)^p}{V(t)}\bigg]^{r / p}w(x) \bigg[\frac{u(x)}{B(x)}\bigg]^q\,dx & < \infty, \\
\int_0^{\infty} \bigg( \int_0^x w(t)\,dt \bigg)^{r / p}\bigg(\sup_{x \le \tau < \infty} \bigg[\frac{u(\tau)}{B(\tau)}\bigg]^p \bigg[\sup_{0 < t \le \tau} \frac{B(t)^p}{V(t)} \bigg]\bigg)^{r / p}w(x)\,dx & < \infty,\\
\int_0^{\infty} \bigg( \int_x^{\infty}\bigg[\sup_{t \le y < \infty} \frac{u(y)^p}{V^2(y)}\bigg]^{q / p} w(t)\,dt\bigg)^{r / p}
\bigg[\sup_{x \le y < \infty} \frac{u(y)^p}{V^2(y)}\bigg]^{q / p} V^{r / p}(x)w(x)\,dx  & < \infty, \\
\int_0^{\infty} \bigg( \int_0^x w(t)\,dt\bigg)^{r / p} \bigg(  \sup_{x \le \tau < \infty}\bigg[\sup_{\tau \le y < \infty} \frac{u(y)^p}{V^2(y)}\bigg] V(\tau)\bigg)^{r / p} w(x)\,dx & \\
= \int_0^{\infty} \bigg( \int_0^x w(t)\,dt\bigg)^{r / p} \bigg(  \sup_{x \le \tau < \infty} \frac{u(\tau)^p}{V(\tau)}\bigg)^{r / p} w(x)\,dx & < \infty.
\end{align*}
Obviously, it remains to show that
\begin{align*}
\int_0^{\infty} \bigg( \int_0^x w(t)\,dt \bigg)^{r / p}\bigg(\sup_{x \le \tau < \infty} \bigg[\frac{u(\tau)}{B(\tau)}\bigg]^p \bigg[\sup_{0 < t \le \tau} \frac{B(t)^p}{V(t)} \bigg]\bigg)^{r / p}w(x)\,dx & \\
& \hspace{-5.5cm} \ls
\int_0^{\infty} \bigg(\int_x^{\infty}  \bigg(\frac{u(t)}{B(t)}\bigg)^q w(t)\,dt\bigg)^{r / p} \bigg[\sup_{0 < \tau \le x} \frac{B(\tau)}{V^{1 / p}(\tau)}\bigg]^{r}w(x) \bigg(\frac{u(x)}{B(x)}\bigg)^q\,dx
\\
& \hspace{-5cm} +
\int_0^{\infty} \bigg( \int_0^x w(t)\,dt\bigg)^{r / p} \bigg(  \sup_{x \le \tau < \infty} \frac{u(\tau)^p}{V(\tau)}\bigg)^{r / p} w(x)\,dx .
\end{align*}
We will prove the assertion only in the case  when $\int_0^{\infty} w(\tau)\,d\tau = \infty$. Let $\{x_k\}$ be such that $\int_0^{x_k} w(\tau)\,d\tau = 2^k$. Then
\begin{align*}
\int_0^{\infty} \bigg( \int_0^x w(t)\,dt \bigg)^{r / p}\bigg(\sup_{x \le \tau < \infty} \bigg[\frac{u(\tau)}{B(\tau)}\bigg]^p \bigg[\sup_{0 < t \le \tau} \frac{B(t)^p}{V(t)} \bigg]\bigg)^{r / p}w(x)\,dx & \\
& \hspace{-5cm} \ap \sum_{k \in \Z} 2^{k r / q} \bigg(\sup_{x_k \le \tau < \infty} \bigg[\frac{u(\tau)}{B(\tau)}\bigg]^p \bigg[\sup_{0 < t \le \tau} \frac{B(t)^p}{V(t)} \bigg]\bigg)^{r / p}.
\end{align*}
Note that
\begin{align*}
\sup_{x_k \le \tau < \infty} \bigg[\frac{u(\tau)}{B(\tau)}\bigg]^p \bigg[\sup_{0 < t \le \tau} \frac{B(t)^p}{V(t)} \bigg] & \ap \sup_{x_k \le \tau < \infty} \bigg[\frac{u(\tau)}{B(\tau)}\bigg]^p \bigg[\sup_{0 < t \le x_k} \frac{B(t)^p}{V(t)} \bigg] +
\sup_{x_k \le \tau < \infty} \bigg[\frac{u(\tau)}{B(\tau)}\bigg]^p \bigg[\sup_{x_k \le t \le \tau} \frac{B(t)^p}{V(t)} \bigg] \\
& =  \bigg[\frac{u(x_k)}{B(x_k)}\bigg]^p \bigg[\sup_{0 < t \le x_k} \frac{B(t)^p}{V(t)} \bigg] +
\sup_{x_k \le t < \infty} \frac{B(t)^p}{V(t)}  \sup_{t \le \tau < \infty} \bigg[\frac{u(\tau)}{B(\tau)}\bigg]^p  \\
& =  \bigg[\frac{u(x_k)}{B(x_k)}\bigg]^p \bigg[\sup_{0 < t \le x_k} \frac{B(t)^p}{V(t)} \bigg] +
\sup_{x_k \le t < \infty} \frac{u(t)^p}{V(t)} \\
& \ap \bigg[\frac{u(x_k)}{B(x_k)}\bigg]^p \bigg[\sup_{0 < t \le x_{k-1}} \frac{B(t)^p}{V(t)} \bigg] + \bigg[\frac{u(x_k)}{B(x_k)}\bigg]^p \bigg[\sup_{x_{k-1} < t \le x_k} \frac{B(t)^p}{V(t)} \bigg] +
\sup_{x_k \le t < \infty} \frac{u(t)^p}{V(t)}
\\
& \le  \bigg[\frac{u(x_k)}{B(x_k)}\bigg]^p \bigg[\sup_{0 < t \le x_{k-1}} \frac{B(t)^p}{V(t)} \bigg] +
\sup_{x_{k-1} \le t < x_k} \frac{u(t)^p}{V(t)} +
\sup_{x_k \le t < \infty} \frac{u(t)^p}{V(t)} \\
& \le  \bigg[\frac{u(x_k)}{B(x_k)}\bigg]^p \bigg[\sup_{0 < t \le x_{k-1}} \frac{B(t)^p}{V(t)} \bigg] +
\sup_{x_{k-1} \le t < \infty} \frac{u(t)^p}{V(t)}.
\end{align*}
Hence,
\begin{align*}
\int_0^{\infty} \bigg( \int_0^x w(t)\,dt \bigg)^{r / p}\bigg(\sup_{x \le \tau < \infty} \bigg[\frac{u(\tau)}{B(\tau)}\bigg]^p \bigg[\sup_{0 < t \le \tau} \frac{B(t)^p}{V(t)} \bigg]\bigg)^{r / p}w(x)\,dx & \\
& \hspace{-7cm} \ls \sum_{k \in \Z} 2^{k r / q} \bigg(\bigg[\frac{u(x_k)}{B(x_k)}\bigg]^p \bigg[\sup_{0 < t \le x_{k-1}} \frac{B(t)^p}{V(t)} \bigg]\bigg)^{r / p} + \sum_{k \in \Z} 2^{k r / q} \bigg(\sup_{x_{k-1} \le t < \infty} \frac{u(t)^p}{V(t)}\bigg)^{r / p} \\
& \hspace{-7cm} \ap \sum_{k \in \Z} \bigg(\int_{x_{k-1}}^{x_k} \bigg(\int_x^{x_k} w\bigg)^{r / p}w(x)\,dx\bigg) \bigg(\bigg[\frac{u(x_k)}{B(x_k)}\bigg]^p \bigg[\sup_{0 < t \le x_{k-1}} \frac{B(t)^p}{V(t)} \bigg]\bigg)^{r / p} \\
&  \hspace{-6.5cm} + \sum_{k \in \Z} \bigg(\int_{x_{k-2}}^{x_{k-1}} \bigg(\int_{x_k}^x w\bigg)^{r / p}w(x)\,dx\bigg) \bigg(\sup_{x_{k-1} \le t < \infty} \frac{u(t)^p}{V(t)}\bigg)^{r / p} \\
& \hspace{-7cm} \ls \sum_{k \in \Z} \int_{x_{k-1}}^{x_k} \bigg(\int_x^{\infty} \bigg(\frac{u(t)}{B(t)}\bigg)^q w(t)\,dt\bigg)^{r / p}\bigg[\sup_{0 < \tau \le x} \frac{B(\tau)}{V^{1 / p}(\tau)}\bigg]^{r}\bigg[\frac{u(x)}{B(x)}\bigg]^q w(x)\,dx \\
&  \hspace{-6.5cm} + \sum_{k \in \Z} \int_{x_{k-1}}^{x_{k+1}} \bigg( \int_0^x w(t)\,dt\bigg)^{r / p} \bigg(  \sup_{x \le \tau < \infty} \frac{u(\tau)^p}{V(\tau)}\bigg)^{r / p} w(x)\,dx \\
& \hspace{-7cm} \ls
\int_0^{\infty} \bigg(\int_x^{\infty}  \bigg(\frac{u(t)}{B(t)}\bigg)^q w(t)\,dt\bigg)^{r / p} \bigg[\sup_{0 < \tau \le x} \frac{B(\tau)}{V^{1 / p}(\tau)}\bigg]^{r}w(x) \bigg(\frac{u(x)}{B(x)}\bigg)^q\,dx \\
& \hspace{-6.5cm} + \int_0^{\infty} \bigg( \int_0^x w(t)\,dt\bigg)^{r / p} \bigg(  \sup_{x \le \tau < \infty} \frac{u(\tau)^p}{V(\tau)}\bigg)^{r / p} w(x)\,dx.
\end{align*}
\end{proof}

\begin{remark}
Note that Theorem \ref{RT.SO.thm.2}, namely the fact that $\eqref{RT.SO.thm.2.eq.2} \Leftrightarrow \eqref{RT.SO.thm.2.eq.1} \Leftrightarrow \eqref{RT.SO.thm.2.eq.3}$, when $b \equiv 1$, was proved in \cite{gogpick2007}.
\end{remark}

As a corollary we obtain that for all the three operators mentioned in \eqref{RT.SO.eq.1.1}, the corresponding weighted inequalities are equivalent. It is worth noticing that this is not so when $p > 1$.
\begin{corollary}\label{RT.SO.thm.1}
    Assume that $0 < p \le 1$, $0 < q < \infty$, and  $v,\,w \in \W\I$. Let $b$ be a weight on $\I$
    such that $0 < B(t) < \infty$ for every $t \in \I$. Then the following three statements are equivalent:

    \begin{align}
    \bigg( \int_0^{\infty} \bigg( \int_0^t f(\tau) b(\tau) \,d\tau\bigg)^q w(t)\,dt\bigg)^{1 / q} & \ls \bigg( \int_0^{\infty} f(\tau)^p v(\tau)\,d\tau\bigg)^{1 / p}, ~ f \in \M^{\dn}; \label{RT.SO.eq.1.2} \\
    \bigg( \int_0^{\infty} \bigg( \int_0^t f(\tau)^p B(\tau)^{p-1}b(\tau) \,d\tau\bigg)^{q / p} w(t)\,dt\bigg)^{1 / q} & \ls \bigg( \int_0^{\infty} f(\tau)^p v(\tau)\,d\tau\bigg)^{1 / p}, ~ f \in \M^{\dn}; \label{RT.SO.eq.1.3} \\
    \bigg( \int_0^{\infty} \bigg( \sup_{0 < \tau \le t} f(\tau) B(\tau) \bigg)^q w(t)\,dt\bigg)^{1 / q} & \ls \bigg( \int_0^{\infty} f(\tau)^p v(\tau)\,d\tau\bigg)^{1 / p}, ~ f \in \M^{\dn}. \label{RT.SO.eq.1.4}
    \end{align}
\end{corollary}
This fact was proved in \cite[Theorem 2.1]{gogpick2007}, when $b \equiv 1$. Recently, in \cite[Theorem 3.9]{GogStep}, it was proved that $\eqref{RT.SO.eq.1.2} \Leftrightarrow \eqref{RT.SO.eq.1.4}$ for more general Volterra operators with continuous Oinarov kernels in the case when $0 < q < p \le 1$.

\begin{proof}
The proof immediately follows from Theorem \ref{RT.SO.thm.2} taking $u \equiv 1$.
\end{proof}

By the way we have proved the following statement.
\begin{theorem}\label{RT.SO.thm.3}
    Let $0 < p \le 1$, $0 < q < \infty$. Assume that $u \in \W\I \cap C\I$ and $b,\,v,\,w \in \W\I$ be such that $0 < V(t) < \infty$ and  $0 < B(t) < \infty$ for all $x > 0$.
    Then inequality \eqref{Tub.thm.1.eq.1}
    is satisfied with the best constant $c$ if and only if:

    {\rm (i)} $p \le q$, and in this case
    $c \ap A_1 + A_2$, where
    \begin{align*}
    A_1: & = \sup_{x > 0} \bigg( \bigg[\sup_{x \le \tau < \infty} \frac{u(\tau)}{B(\tau)}\bigg]^q \int_0^x w(t)\,dt + \int_x^{\infty} \bigg[\sup_{t \le \tau < \infty} \frac{u(\tau)}{B(\tau)}\bigg]^q w(t)\,dt\bigg)^{1 / q} \sup_{0 < y \le x} \frac{B(y)}{V^{1 / p}(y)}; \\
    A_2: & = \sup_{x > 0}\bigg( \bigg[ \sup_{x \le y < \infty} \frac{u(y)^p}{V^2(y)}\bigg]^{q / p}  \int_0^x w(t)\,dt + \int_x^{\infty}  \bigg[ \sup_{t \le y < \infty} \frac{u(y)^p}{V^2(y)}\bigg]^{q / p} w(t)\,dt\bigg)^{1 / q} V^{1 / p}(x);
    \end{align*}

    {\rm (ii)} $q < p$,  and in this case $c \ap B_1 + B_2 + B_3 + B_4$, where
    \begin{align*}
    B_1: & = \bigg(\int_0^{\infty} \bigg(\int_0^x w(t)\,dt\bigg)^{r / p} \bigg[\sup_{x \le \tau < \infty} \bigg[\sup_{\tau \le y < \infty} \frac{u(y)}{B(y)}\bigg]^q  \bigg( \sup_{0 < y \le x} \frac{B(y)}{V^{1 / p}(y)} \bigg)\bigg]^r w(x)\,dx \bigg)^{1 / r}, \\
    B_2: & = \bigg(\int_0^{\infty} \bigg(\int_x^{\infty} \bigg[\sup_{t \le \tau < \infty} \frac{u(\tau)}{B(\tau)}\bigg]^q w(t)\,dt\bigg)^{r / p} \bigg[\sup_{0 < \tau \le x} \frac{B(\tau)}{V^{1 / p}(\tau)}\bigg]^{r}\bigg[\sup_{x \le \tau < \infty} \frac{u(\tau)}{B(\tau)}\bigg]^q w(x)\,dx \bigg)^{1 / r}, \\
    B_3: & = \bigg( \int_0^{\infty} \bigg( \int_0^x w(t)\,dt\bigg)^{r / p} \bigg(  \sup_{x \le \tau < \infty}\bigg[\sup_{\tau \le y < \infty} \frac{u(y)^p}{V^2(y)}\bigg] V(\tau)\bigg)^{r / p} w(x)\,dx \bigg)^{1 / r}, \\
    B_4: & = \bigg(\int_0^{\infty} \bigg( \int_x^{\infty}\bigg[\sup_{t \le y < \infty} \frac{u(y)^p}{V^2(y)}\bigg]^{q / p} w(t)\,dt\bigg)^{r / p}
    \bigg[\sup_{x \le y < \infty} \frac{u(y)^p}{V^2(y)}\bigg]^{q / p} V^{r / p}(x)w(x)\,dx \bigg)^{1 / r}.
    \end{align*}
\end{theorem}

\end{document}